\documentclass[11pt]{amsart}
\usepackage{fullpage,hyperref,enumitem}
\usepackage{tikz,color}

\newtheorem{theorem}{Theorem}[section]
\newtheorem{proposition}[theorem]{Proposition}
\newtheorem{lemma}[theorem]{Lemma}

\theoremstyle{definition}
\newtheorem{example}[theorem]{Example}  
\newtheorem{definition}[theorem]{Definition}   

\theoremstyle{remark}

\renewcommand{\S}{\mathcal{S}}
\newcommand{\M}{\mathcal{M}}
\renewcommand{\L}{\mathcal{L}}
\newcommand{\W}{\mathcal{W}}
\newcommand{\So}{\S^o}
\newcommand{\Se}{\S^e}
\newcommand{\Mo}{\M^o}
\newcommand{\Me}{\M^e}
\newcommand{\Lo}{\L^o}
\newcommand{\Le}{\L^e}
\newcommand{\Wo}{\W^o}
\newcommand{\We}{\W^e}
\renewcommand{\u}{\mathtt{u}}
\renewcommand{\v}{\mathtt{v}}
\newcommand{\w}{\mathtt{w}}
\renewcommand{\a}{\mathtt{a}}
\renewcommand{\b}{\mathtt{b}}
\renewcommand{\c}{\mathtt{c}}
\renewcommand{\d}{\mathtt{d}}
\newcommand{\e}{\mathtt{e}}

\DeclareMathOperator{\wt}{wt}
\DeclareMathOperator{\Des}{Des}
\DeclareMathOperator{\Asc}{Asc}
\newcommand{\lalt}{<_{\mathrm{alt}}}
\newcommand{\x}{\mathbf{x}}
\newcommand{\Psinv}{\Omega}
\newcommand{\psinv}{\omega}
\newcommand{\cS}{\texttt{(S)}}
\newcommand{\cP}{\texttt{(P)}}
\newcommand{\cF}{\texttt{(F)}}
\newcommand{\cSi}{\texttt{(S')}}
\newcommand{\cPi}{\texttt{(P')}}
\newcommand{\cFi}{\texttt{(F')}}
\newcommand\dashline{\rotatebox[origin=c]{90}{\,-\,-}}

\title{A bijection for descent sets of permutations with only even and only odd cycles}

\author{Sergi Elizalde}


\begin{document}

\maketitle

\begin{abstract}
It is known that, when $n$ is even, the number of permutations of $\{1,2,\dots,n\}$ all of whose cycles have odd length equals the number of those all of whose cycles have even length. Adin, Heged\H{u}s and Roichman recently found a surprising refinement of this identity. 
They showed that, for any fixed set $J$, 
the equality still holds when restricting to permutations with descent set $J$ on one side, and permutations with ascent set $J$ on the other.
Their proof uses generating functions for higher Lie characters, and it also yields a version for odd~$n$.

Here we give a bijective proof of their result. We first use known bijections, due to Gessel, Reutenauer and others, to restate the identity in terms of multisets of necklaces, which we interpret as words, and then describe a new weight-preserving bijection between words all of whose Lyndon factors have odd length and are distinct, and words all of whose Lyndon factors have even length.
We also show that the corresponding equality about Lyndon factorizations has a short proof using generating functions.
\end{abstract}

\section{Introduction}\label{sec:intro}

For a positive integer $n$, let $\S_n$ denote the symmetric group on $[n]=\{1,2,\dots,n\}$. Let $\So_n$ be the set of permutations in $\S_n$ all of whose cycles have odd length. 
Let $\Se_n$ be the set of permutations all of whose cycles have even length, except possibly for one cycle of length one (i.e., a fixed point). Note that permutations in $\Se_n$ consist of only cycles of even length if $n$ is even, and they have one fixed point if $n$ is odd.

It is known that
$$|\So_n|=|\Se_n|=\begin{cases} (n-1)!!^2= (n-1)^2(n-3)^2\dots  1^2 & \text{if $n$ is even},\\
n\,(n-2)!!^2=n (n-2)^2(n-4)^2\dots 2^2& \text{if $n$ is odd}.
\end{cases}$$
This is proved, for example, in B\'ona's book \cite[Thms.\ 6.24 \& 6.25]{Bona}, and it can also easily be shown using exponential generating functions.
For even $n$, the following bijective proof of the equality $|\So_n|=|\Se_n|$ appears in \cite[Lem.~6.20]{Bona}. Given a permutation in $\So_n$, first write it in standard cycle form, where each cycle starts with its largest element, and cycles are ordered by increasing first element. Let the cycles be $C_1,C_2,\dots,C_{2k}$ from left to right. Then, for each $i\in[k]$, take the last element from $C_{2i-1}$ and place it at the end of $C_{2i}$. As shown in~\cite{Bona}, this construction is a bijection from $\So_n$ to $\Se_n$ when $n$ is even. It is not hard to extend this bijection to the case of odd $n$.

In a recent preprint~\cite{AHR}, Adin, Heged\H{u}s and Roichman proved that that the identity $|\So_n|=|\Se_n|$ has a surprising refinement. 
For $\pi\in\S_n$, denote its descent set by $\Des(\pi)=\{i\in[n-1]:\pi_i>\pi_{i+1}\}$, and its ascent set by $\Asc(\pi)=\{i\in[n-1]:\pi_i<\pi_{i+1}\}$. By definition, $\Asc(\pi)=[n-1]\setminus\Des(\pi)$. 
The following is the main result from~\cite{AHR}, stated here in an equivalent form where the roles of ascents and descents are switched; the reason for this will become clear later.

\begin{theorem}[\cite{AHR}]\label{thm:AHR}
For any positive integer $n$ and any subset $J\subseteq [n-1]$,
$$|\{\pi\in\So_n:\Asc(\pi)=J\}| = |\{\pi\in\Se_n: \Des(\pi)=J\}|.$$
\end{theorem}

Adin, Heged\H{u}s and Roichman show that Theorem~\ref{thm:AHR} is equivalent to an identity for higher Lie characters on $\S_n$, and then use heavy machinery to find generating functions for these characters. A natural question is whether Theorem~\ref{thm:AHR} has a bijective proof. Unfortunately, B\'ona's bijection between $\So_n$ and $\Se_n$ does not behave well with respect to descent sets, and neither does any simple variation of it.

The goal of this paper is to give a combinatorial proof of Theorem~\ref{thm:AHR}. More specifically, we will prove the following.

\begin{theorem}\label{thm:main}
For any positive integer $n$ and any subset $S\subseteq [n-1]$, there exists an explicit bijection
\begin{equation}\label{eq:main}
f_S: \{\pi\in\So_n:\Asc(\pi)\subseteq S\}\to\{\pi\in\Se_n: \Des(\pi)\subseteq S\}.
\end{equation}
\end{theorem}

The proof combines two known bijections between permutations and multisets of necklaces, together with a new bijection for Lyndon factorizations of words. We will show that the identity for Lyndon factorizations also has an alternative short proof using generating functions. 

Our bijection proves a slight modification of Theorem~\ref{thm:AHR}, namely, that 
\begin{equation}\label{eq:OEsubset}
|\{\pi\in\So_n:\Asc(\pi)\subseteq S\}| = |\{\pi\in\Se_n: \Des(\pi)\subseteq S\}|
\end{equation}
for any $S\subseteq [n-1]$. Theorem~\ref{thm:AHR} follows from this equality by the principle of inclusion-exclusion, since
\begin{align*}|\{\pi\in\So_n:\Asc(\pi)=J\}|&=\sum_{S\subseteq J}(-1)^{|J\setminus S|} |\{\pi\in\So_n:\Asc(\pi)\subseteq S\}\\
&=\sum_{S\subseteq J}(-1)^{|J\setminus S|} |\{\pi\in\Se_n: \Des(\pi)\subseteq S\}|=|\{\pi\in\Se_n: \Des(\pi)=J\}|.\end{align*}

The paper is structured as follows. In Section~\ref{sec:necklaces} we describe two bijections between permutations and multisets of necklaces. The first one is a classical bijection due to Gessel and Reutenauer~\cite{GR}, which lies at the heart of algebraic combinatorics. We apply it to permutations with even cycles whose descent set is contained in $S$. 
The second bijection is a variation of the first, and is a special case of a bijection of Gessel, Restivo and Reutenauer~\cite{GRR}, and of a more general bijection of Steinhardt~\cite{Steinhardt}.
We apply it to permutations with odd cycles whose ascent set is contained in $S$.
In Section~\ref{sec:GF} we interpret the resulting multisets of necklaces as words whose Lyndon factorization has only even factors on one hand, and 
words whose Lyndon factorization has only odd and distinct factors on the other hand. 
We first use generating functions to show that both sets of words have the same cardinality, which gives relatively short proof of Theorem~\ref{thm:AHR}.
In Section~\ref{sec:bijection} we construct an explicit bijection between the two sets of words, completing the proof of Theorem~\ref{thm:main}. The description of the map is relatively simple, but proving that it is a bijection takes some work.

\section{From permutations to multisets of necklaces}\label{sec:necklaces}

In this section we show that both sides of equation~\eqref{eq:main} are in bijection with multisets of primitive necklaces satisfying certain conditions. Let us start with some definitions and notation.

Let $S=\{s_1,s_2,\dots,s_{k-1}\}\subseteq[n-1]$, where $s_1<s_2<\dots<s_{k-1}$. Denote its associated composition by $\alpha=\alpha(S)=(s_1,s_2-s_1,\dots,s_{k-1}-s_{k-2},n-s_{k-1})$, and define the monomial $\x^{\alpha}=\prod_{i=1}^k x_i^{\alpha_i}$.

Fix an alphabet $A=\{a_1,a_2,\dots,a_k\}$, with total order $a_1<a_2<\dots<a_k$. Sometimes we will denote the letters by $\a<\b<\c<\cdots$ instead. The assumption that the alphabet is finite will simplify our notation, but this assumption is not important. Denote by $\W=A^\ast$ the set of finite words over $A$, and by $\W_n$ the set of those of length $n$. Define two words $u,v\in\W$ to be {\em conjugate} if they are cyclic rotations of each other, that is, there exist words $r$ and $s$ such that $u=rs$ and $v=sr$.
A {\em necklace} is a conjugacy class of words in $\W$. 
A nonempty word $u$ is {\em primitive} if it is not the power of another word, i.e., it is not of the form $u=r^j$ for $j\ge2$.
A necklace is {\em primitive} if it is the conjugacy class of a primitive word.

Let $\M_n$ be the set consisting of all multisets of primitive necklaces of total length $n$. Given $M\in\M_n$, its {\em cycle structure} is the partition of $n$ whose parts are the lengths of the necklaces in the multiset, and its {\em weight} is the monomial $\wt(M)=x_1^{\alpha_1}x_2^{\alpha_2}\dots x_k^{\alpha_k}$, where $\alpha_i$ is the number of times that $a_i$ appears in $M$.

\subsection{Descent sets contained in $S$}

In \cite[Lem.~3.4]{GR}, Gessel and Reutenauer describe a bijection $U$ (later denoted by $\Phi$ in~\cite{GRR}) between words and multisets of necklaces. For our purposes, it will be more convenient (and simpler) to interpret it as a map on permutations; specifically, as a bijection
$$\Phi_S:\{\pi\in\S_n:\Des(\pi)\subseteq S\}\to\{M\in\M_n:\wt(M)=\x^{\alpha(S)}\}$$ that preserves the cycle structure. Here is a description of $\Phi_S$.

Given $\pi\in\S_n$ with $\Des(\pi)\subseteq S$, write $\pi$ in cycle form, and replace entries $1,\dots,s_1$ with $a_1$, entries $s_1+1,\dots,s_2$ with $a_2$, and so on, finally replacing entries $s_{k-1}+1,\dots,n$ with $a_k$. This operation turns each cycle of $\pi$ into a necklace.
Let $\Phi_S(\pi)$ be the resulting multiset of necklaces, and note that it was weight $\x^{\alpha(S)}$. It is shown in \cite{GR} that these necklaces  are primitive, and that $\Phi_S$ is a bijection.

The inverse map can be described as follows. Given $M\in\M_n$ with $\wt(M)=\x^{\alpha(S)}$, label each element of each necklace by the periodic sequence obtained by reading the necklace starting at that element. Then, order these sequences lexicographically; if some necklaces appear with multiplicity, break the ties in a consistent way by first ordering the repeated necklaces. This order assigns a number from $1$ to $n$ to each element of each necklace, which produces the permutation $\Phi_S^{-1}(M)$ in cycle form.

In order to describe the image of $\Phi_S$ when restricted to the right-hand side of equation~\eqref{eq:main}, let $\Me_n$ be the set of elements in $\M_n$ which consist of necklaces of even length only, except possibly for one necklace of length one. Since the bijection $\Phi_S$ preserves the cycle structure, we obtain the following result.

\begin{proposition}\label{prop:Phi}
The map $\Phi_S$ defined above restricts to a bijection $$\Phi_S:\{\pi\in\Se_n: \Des(\pi)\subseteq S\}\to\{M\in\Me_n:\wt(M)=\x^{\alpha(S)}\}.$$
\end{proposition}

\begin{example}\label{ex:Phi}
Let $S=\{4,7\}$, and let $\pi=45672381\in\Se_8$, which has $\Des(\pi)=\{4,7\}\subseteq S$. 
Writing $\pi$ in cycle form as $(3,6)(2,5)(1,4,7,8)$ and replacing entries $1,2,3,4$ with $\a$, entries $5,6,7$ with $\b$, and entry $8$ with $\c$, we obtain the multiset of necklaces
$\Phi_S(\pi)=(\a,\b)(\a,\b)(\a,\a,\b,\c)$.

To recover $\pi$ from $\Phi_S(\pi)$, replace the elements of the necklaces by their periodic labels,
$$(\a\b\a\b\dots,\b\a\b\a\dots)(\a\b\a\b\dots,\b\a\b\a\dots)(\a\a\b\c\dots,\a\b\c\a\dots,\b\c\a\a\dots,\c\a\a\b\dots)$$
and order these labels lexicographically, that is,
$$\a\a\b\c\dots<\a\b\a\b\dots=\a\b\a\b\dots<\a\b\c\a\dots<\b\a\b\a\dots=\b\a\b\a\dots<\b\c\a\a\dots<\c\a\a\b\dots.$$
Consistently breaking the ties in the repeated necklace, we recover
$\pi=(3,6)(2,5)(1,4,7,8)$.
\end{example}

\subsection{Ascent sets contained in $S$}

The bijection $\Phi_S$ does not work well with permutations whose ascent set is contained in $S$. Instead, to deal with the left-hand side of equation~\eqref{eq:main}, we will use a different bijection that has appeared in work of Gessel, Restivo and Reutenauer~\cite{GRR}, and is also a special case of a bijection due to Steinhardt~\cite{Steinhardt}.

In \cite[Sec.~3]{GRR}, the authors describe a bijection $\Xi$ between words of length $n$ and multisets of necklaces of total length $n$ that satisfy two conditions: each necklace is either primitive or of the form $(uu)$ for some primitive word $u$ of odd length, and each necklace of odd length appears with multiplicity at most one. As we did for $\Phi$, it will be convenient to interpret $\Xi$ as a map on permutations. For our purposes, it suffices to describe the map on permutations that consist of only odd cycles. Because of the above conditions, the resulting multisets of necklaces consist of primitive necklaces of odd length, all of which are distinct. This greatly simplifies the construction, and it is the reason that, in our statement of Theorem~\ref{thm:AHR}, we switched the roles of $\Asc$ and $\Des$ with respect to~\cite{AHR}. Denote by $\Mo_n$ the set of elements in $\M_n$ which consist of distinct necklaces of odd length.

A slight modification of $\Xi$ gives a bijection
$$\Xi_S:\{\pi\in\So_n: \Asc(\pi)\subseteq S\}\to\{M\in\Mo_n:\wt(M)=\x^{\alpha(S)}\}$$
that preserves the cycle structure. The description of $\Xi_S$ that we give next is almost identical to that of  $\Phi_S$, but the key difference appears when describing their inverses.
We point out that the connection to ascent sets of permutations is not mentioned in~\cite{GRR}, but it is in~\cite{Steinhardt}.

Given $\pi\in\So_n$ with $\Asc(\pi)\subseteq S$, write $\pi$ in cycle form, and replace entries $1,\dots,s_1$ with $a_1$, 
entries $s_1+1,\dots,s_2$ with $a_2$, and so on, finally replacing entries $s_{k-1}+1,\dots,n$ with $a_k$.
Let $\Xi_S(\pi)$ be the resulting multiset of necklaces, which has weight $\x^{\alpha(S)}$. Since $\pi\in\So_n$, all the necklaces in $\Xi_S(\pi)$ have odd length, which implies, as shown in~\cite{GRR,Steinhardt}, that they are primitive and distinct. It follows that $\Xi_S(\pi)\in\Mo_n$.

To describe the inverse map, we first need to define the {\em alternating lexicographic order} on words, denoted by $\lalt$.
It is defined by $\u_1\u_2\u_3\dots\lalt \v_1\v_2\v_3\dots$ if either $\u_1<\v_1$, or $\u_1=\v_1$ and $\v_2\v_3\dots\lalt \u_2\u_3\dots$.

Given $M\in\Mo_n$ with $\wt(M)=\x^{\alpha(S)}$, label each element of each necklace by the periodic sequence obtained by reading the necklace starting at that element. Then, order these sequences according to $\lalt$. Note that there will be no ties, since all the necklaces in $M$ are primitive and distinct. This order assigns a number from $1$ to $n$ to each element of each necklace, which produces the permutation $\Xi_S^{-1}(M)$ in cycle form.

The fact that these constructions are inverses from each other is proved in~\cite{GRR,Steinhardt}. We deduce the following.

\begin{proposition}\label{prop:Xi}
The map $\Xi_S$ defined above is a bijection $$\Xi_S:\{\pi\in\So_n: \Asc(\pi)\subseteq S\}\to\{M\in\Mo_n:\wt(M)=\x^{\alpha(S)}\}.$$
\end{proposition}

\begin{example}\label{ex:Xi}
Let $S=\{4,7\}$, and let $\pi=86325417\in\So_8$, which has $\Asc(\pi)=\{4,7\}\subseteq S$. 
Writing $\pi$ in cycle form as $(5)(3)(2,6,4)(1,8,7)$ and replacing entries $1,2,3,4$ with $\a$, entries $5,6,7$ with $\b$, and entry $8$ with $\c$, we obtain the multiset of necklaces
$\Xi_S(\pi)=(\b)(\a)(\a,\b,\a)(\a,\c,\b)$.

To recover $\pi$ from $\Xi_S(\pi)$, replace the elements of the necklaces by their periodic labels,
$$(\b\b\b\dots)(\a\a\a\dots)(\a\b\a\dots,\b\a\a\dots,\a\a\b\dots)(\a\c\b\dots,\c\b\a\dots,\b\a\c\dots)$$
and order these labels according to $\lalt$, that is,
$$\a\c\b\dots\lalt \a\b\a\dots\lalt \a\a\a\dots\lalt \a\a\b\dots\lalt \b\b\b\dots\lalt \b\a\a\dots\lalt \b\a\c\dots\lalt \c\b\a\dots,$$
to recover
$\pi=(5)(3)(2,6,4)(1,8,7)$.
\end{example}

Combining Propositions~\ref{prop:Phi} and~\ref{prop:Xi}, equation~\eqref{eq:OEsubset} is equivalent to the equality
\begin{equation}\label{eq:OEnecklaces}
|\{M\in\Mo_n:\wt(M)=\x^{\alpha(S)}\}|=|\{M\in\Me_n:\wt(M)=\x^{\alpha(S)}\}|.
\end{equation}
In Section~\ref{sec:GF} we will give a short proof of this equality using generating functions. In Section~\ref{sec:bijection} we will give a bijective proof.

\section{Lyndon factorizations and generating functions}\label{sec:GF}

In this section we identify primitive necklaces with Lyndon words, which allows us to view multisets of primitive necklaces as Lyndon factorizations of words.
In the rest of the paper, we use $<$ to denote the lexicographic order on $\W$. The terms smaller and larger, when applied to words, will always refer to lexicographically smaller and larger, whereas we will use the adjectives shorter and longer when comparing lengths.

A primitive word in $\W$ is called a {\em Lyndon word} if it is smaller than all the other words in its conjugacy class; equivalently, if it is smaller than all of its proper suffixes~\cite[Prop.~5.1.2]{Lothaire}. Lyndon words are in one-to-one correspondence with primitive necklaces, since each conjugacy class of primitive words has a unique lexicographically smallest element.
Denote by $\L$ the set of Lyndon words in $\W$. We write $\L=\Lo\sqcup\Le$, where $\Lo$ and $\Le$ are the subsets consisting of words of odd and even length, respectively. From now on, we will say that a word is {\em odd} or {\em even} to mean that it has odd or even length, respectively.

The following is a well-known result of Lyndon.

\begin{theorem}[{\cite[Thm.~5.1.5]{Lothaire}}]\label{thm:Lyndon}
Every $w\in\W$ has a unique {\em Lyndon factorization}, that is, an expression $w=\ell_1\ell_2\dots \ell_m$ where $\ell_i\in\L$ for all $i$, and $\ell_1\ge\ell_2\ge\dots\ge\ell_m$. 
\end{theorem}
 
We will use vertical bars to indicate that $w=\ell_1|\ell_2|\dots|\ell_m$ is the Lyndon factorization of $w$, and call the words $\ell_i$ the {\em Lyndon factors} of $w$. By identifying primitive necklaces with Lyndon words, Theorem~\ref{thm:Lyndon} gives a straightforward bijection between $\M_n$ and $\W_n$, where each necklace in $M\in\M_n$ becomes a Lyndon factor of the associated word $w\in\W_n$. 
For instance, the multiset $(\a,\b)(\a,\b)(\a,\a,\b,\c)$ in Example~\ref{ex:Phi} corresponds to the word $\a\b|\a\b|\a\a\b\c$, and the multiset 
$(\b)(\a)(\a,\b,\a)(\a,\c,\b)$ in Example~\ref{ex:Xi} corresponds to the word $\b|\a\c\b|\a\a\b|\a$, where the vertical bars indicate their unique Lyndon factorizations. 

As we did for multisets of necklaces, we define the {\em weight} of a word $w\in \W$ to be the monomial $\wt(w)=x_1^{\alpha_1}x_2^{\alpha_2}\dots x_k^{\alpha_k}$, where $\alpha_i$ is the number of times that $a_i$ appears in $w$. The {\em length} of $w$ is $|w|=\alpha_1+\alpha_2+\dots+\alpha_k$.
Note that if $w=\ell_1\ell_2\dots \ell_m$, then $\wt(w)=\wt(\ell_1)\wt(\ell_2)\cdots \wt(\ell_m)$. 

Let $\Wo_n$ be the set of words in $\W_n$ all of whose Lyndon factors have odd length and are distinct. Let $\We_n$ be the set of words in $\W_n$ all of whose Lyndon factors have even length, except possibly for one factor which has length one (note that a factor of length one occurs precisely when $n$ is odd). 

The above correspondence between mutlisets of necklaces and words gives straightforward bijections between $\Mo_n$ and $\Wo_n$, and between $\Me_n$ and $\We_n$. With this identification, Equation~\eqref{eq:OEnecklaces} is equivalent to the following equality, which we prove using generating functions.

\begin{proposition}\label{prop:Lyndon}
For any composition $\alpha$ of $n$,
$$|\{w\in\Wo_n:\wt(w)=\x^\alpha\}|=|\{w\in\We_n:\wt(w)=\x^\alpha\}|.$$
\end{proposition}

\begin{proof}
The generating function for the set $\Wo=\bigcup_{n\ge0}\Wo_n$ of words whose Lyndon factors are odd and distinct is
$$\sum_{w\in\Wo} \wt(w)=\prod_{\ell\in\Lo} \left(1+\wt(\ell)\right).$$
The generating function for the set $\We=\bigcup_{n\ge0}\We_n$ of words whose Lyndon factors are even, except for possibly one factor of length one, is
$$\sum_{w\in\We} \wt(w)=(1+x_1+x_2+\dots+x_k)\prod_{\ell\in\Le} \frac{1}{1-\wt(\ell)}.$$

Thus, the statement that we want to prove can be rephrased in terms of generating functions as
\begin{equation}\label{eq:Lyndon-GF}
\prod_{\ell\in\Lo} \left(1+\wt(\ell)\right)=(1+x_1+x_2+\dots+x_k)\prod_{\ell\in\Le} \frac{1}{1-\wt(\ell)},
\end{equation}
or equivalently,
\begin{equation}\label{eq:parity_number_of_factors}
\prod_{\ell\in\Lo} \left(1+\wt(\ell)\right)\prod_{\ell\in\Le} \left(1-\wt(\ell)\right)=1+x_1+x_2+\dots+x_k.
\end{equation}
Substituting $x_i$ for $-x_i$ for $1\le i\le k$, and noting that this change reverses the sign of $\wt(\ell)$ if $\ell$ has odd length but preserves it otherwise, we can rewrite~\eqref{eq:parity_number_of_factors} as
$$\prod_{\ell\in\Lo} \left(1-\wt(\ell)\right)\prod_{\ell\in\Le} \left(1-\wt(\ell)\right)=1-x_1-x_2-\dots-x_k$$
Taking reciprocals, this is equivalent to
$$\prod_{\ell\in\L} \frac{1}{1-\wt(\ell)}=\frac{1}{1-x_1-x_2-\dots-x_k},$$
which holds because every word in $\W$ has a unique Lyndon factorization by Theorem~\ref{thm:Lyndon}.
\end{proof}

\section{A bijection from odd and distinct to even Lyndon factorizations}\label{sec:bijection}

In this section we complete the proof of Theorem~\ref{thm:main} by giving a bijective proof of Proposition~\ref{prop:Lyndon}.
In Subsection~\ref{sec:psi} we describe a weight-preserving bijection $\Psi:\Wo_n\to\We_n$, in Subsection~\ref{sec:psinv} we describe its inverse $\Psinv$, and in Subsections~\ref{sec:auxiliary} and~\ref{sec:properties} we prove that $\Psi$ and $\Psinv$ are indeed inverses of each other.

\subsection{The bijection $\Psi$}\label{sec:psi}

The definition of $\Psi$ uses the notion of standard factorization of a Lyndon word, which we describe next. A proper suffix of $w$ is a suffix that is not empty and not equal to $w$.

\begin{lemma}[{\cite[Prop.~5.1.3]{Lothaire}}]\label{lem:standard} A word $w\in\W$ is a Lyndon word if and only if $|w|=1$ or $w=rs$ with $r,s\in\L$ and $r<s$. 
Additionally, if $w\in\L$ with $|w|\ge2$, and $s$ is the longest proper suffix of $w$ that belongs to $\L$, then $r\in\L$ and $r<rs<s$. 
\end{lemma}

When $s$ is the longest proper suffix of $w$ that belongs to $\L$, the expression $w=rs$ is called the {\em standard factorization} of $w$. 
We will denote this by $w=r\dashline s$. A useful equivalent characterization of $s$ is that it is also the (lexicographically) smallest proper suffix of $w$. Indeed, being the smallest guarantees that $s\in\L$, and that no longer proper suffix belongs to $\L$. 

Given $w\in\Wo_n$, we will build $\Psi(w)$ by repeatedly applying certain updates to a pair of words $(O,E)$. Initially, $(O,E)=(w,-)$, where $-$ denotes the empty word. Each step moves some subword from $O$ to the beginning of $E$. At any time, all the Lyndon factors of $O$ are odd and distinct, and all the Lyndon factors of $E$ are even. At the end of the algorithm, we have $(O,E)=(-,\Psi(w))$. For convenience, define a word, denoted by $\infty$, which satisfies $w<\infty$ for any $w\in\W$.

\begin{definition}[The map $\Psi$]\label{def:psi} On input $w\in\Wo_n$, initially set $(O,E)=(w,-)$, and iterate the following step as long as $|O|\ge2$:
\begin{itemize}
\item [$(\psi)$]
Let $O=o_1|o_2|\dots|o_m$ be the Lyndon factorization of $O$. 
Say that $o_m$ is {\em splittable} if $|o_m|\ge2$ and its standard factorization $o_m=r\dashline s$ satisfies $s<o_{m-1}$ (with the convention $o_{m-1}=\infty$ if $m=1$).
Update $(O,E)$ to
$$(O',E')=\begin{cases}
(o_1o_2\dots o_{m-1}r,\,sE) & \text{if $o_m$ is splittable and $r$ is odd}, \hfill \quad\cS\\
(o_1o_2\dots o_{m-1}s,\,rE) & \text{if $o_m$ is splittable and $r$ is even}, \hfill \quad\cP\\
(o_1o_2\dots o_{m-2},\,o_m o_{m-1} E) & \text{if $o_m$ is not splittable}. \hfill \quad\cF
\end{cases}$$
\end{itemize}

If we reach $|O|=1$ (this case only occurs when $n$ is odd), move this letter to the Lyndon factorization of $E$ by inserting it as a new factor, in the unique location that keeps the factors weakly decreasing from left to right.

Once $O$ is empty, let $\Psi(w)=E$.
\end{definition}

The steps of type \cS, \cP\ and \cF\ are named after {\em suffix}, {\em prefix} and {\em flip}, respectively. 
Note that, since $o_m=rs$ has odd length, $r$ is odd if and only if $s$ is even.
In the rest of this section we prove the following theorem.

\begin{theorem}\label{thm:bijection}
The map $\Psi:\Wo_n\to\We_n$ is a weight-preserving bijection.
\end{theorem}

Let us first give a few examples of the map $\Psi$. An implementation in SageMath of this bijection, as well as of its inverse, is available in~\cite{code}.

\begin{example}\label{ex:10}
Let $w=\d\a\d\c\c\d\b\c\c\c\in\Wo_{10}$, with Lyndon factorization $o_1|o_2=\d|\a\d\c\c\d\b\c\c\c$.
The rightmost factor has standard factorization $o_2=r\dashline s=\a\d\c\c\d\dashline \b\c\c\c$. Since $s<o_1$, the word $o_2$ is splittable, and since $s$ is even, we apply~\cS\ and move it to~$E$.

Now $O=o_1|o_2=\d|\a\d\c\c\d$, and $o_2=r\dashline s=\a\d\dashline \c\c\d$. Again $s<o_1$, so $o_2$ is splittable, and since $r$ is even, we apply~\cP\ and move it to $E$.

In the third iteration,  $O=o_1|o_2=\d|\c\c\d$, and $o_2=r\dashline s=\c\dashline \c\d$. Again $s<o_1$, so $o_2$ is splittable, and since $s$ is even, we apply~\cS\ and move it to $E$.

Now $O=\d|\c$, and $o_2=\c$ is not splittable because it has length one. Applying~\cF, we move $o_2o_1=\c\d$ to $E$, and we obtain $\Psi(w)=\c\d\c\d\a\d\b\c\c\c$. We can summarize these steps in a table as follows.
$$\begin{array}{lcr}
  O && E\\ \hline
 \d|\a\d\c\c\d\dashline \b\c\c\c && - \\
  \d|\a\d\dashline \c\c\d &^\cS& \b\c\c\c \\
  \d|\c\dashline \c\d &^\cP& \a\d\b\c\c\c \\
  \d|\c &^\cS& \c\d\a\d\b\c\c\c \\
  -&^\cF& \c\d\c\d\a\d\b\c\c\c 
\end{array}$$
\end{example}

The Lyndon factorization $\Psi(w)=\c\d|\c\d|\a\d\b\c\c\c$ in the above example illustrates that the subwords that are moved to $E$ at each step do not necessarily become its Lyndon factors. In general, each Lyndon factor of $E$ may consist of more than one of these subwords. This will be relevant in the next sections when we describe the inverse map.

\begin{example}\label{ex:8}
For $w=\b\a\b\a\c\a\b\c\in\Wo_8$, the following table summarizes the computation of $\Psi(w)=\a\b\c\b\a\b\a\c$.
$$\begin{array}{lcr}
O && E\\ \hline
\b|\a\b\a\c\dashline \a\b\c && - \\
 \b|\a\dashline \b\c &^\cP& \a\b\a\c \\
- &^\cF& \a\b\c\b\a\b\a\c
\end{array}$$
Note that, in the second step, the word $o_2=\a\dashline \b\c$ is not splittable because $s=\b\c\ge \b=o_1$.
\end{example}

\begin{example}\label{ex:31}
For $w=\b\c\c\c\b\b\c\c\b\b\c\c\b\c\b\a\b\a\b\a\a\b\c\a\a\a\b\b\a\a\b\in\Wo_{31}$, the following table shows the steps in the computation of $\Psi(w)=\b\c\c\c\b\b\c\c\b\b\c\c\b\c\b\a\b\a\b\a\a\b\c\a\a\a\b\b\a\a\b$.
$$\begin{array}{lcr}
O && E\\ \hline
\b\b\c\c\b\b\c\c\c\b\b\c\c\b\c|\b|\a\a\b\a\a\b\c|\a\a\b|\a\dashline \a\a\b\b && - \\
\b\b\c\c\b\b\c\c\c\b\b\c\c\b\c|\b|\a\a\b\dashline \a\a\b\c\ &^\cF& \a\a\a\b\b\a\a\b \\
\b\b\c\c\b\b\c\c\c\b\b\c\c\b\c|\b|\a\dashline \a\b\ &^\cS& \a\a\b\c\a\a\a\b\b\a\a\b \\
\b\b\c\c\b\b\c\c\c\b\b\c\c\b\c|\b|\a\ &^\cS& \a\b\a\a\b\c\a\a\a\b\b\a\a\b \\
\b\b\c\c\b\b\c\c\c\dashline \b\b\c\c\b\c &^\cF& \a\b\a\b\a\a\b\c\a\a\a\b\b\a\a\b \\
\b\b\c\c\dashline \b\b\c\c\c  &^\cS& \b\b\c\c\b\c\a\b\a\b\a\a\b\c\a\a\a\b\b\a\a\b \\
\b\dashline \b\c\c\c  &^\cP& \b\b\c\c\b\b\c\c\b\c\a\b\a\b\a\a\b\c\a\a\a\b\b\a\a\b \\
\b  &^\cS& \b\c\c\c\b\b\c\c\b\b\c\c\b\c\a\b\a\b\a\a\b\c\a\a\a\b\b\a\a\b \\
- && \b\c\c\c|\b\b\c\c\b\b\c\c\b\c|\b|\a\b|\a\b|\a\a\b\c|\a\a\a\b\b\a\a\b
\end{array}$$
\end{example}

\begin{example}\label{ex:17}
For $w=\d\d\e\c\e\d\b\d\b\d\c\c\d\a\b\d\a\in\Wo_{17}$, we compute $\Psi(w)=\d\e\d\c\c\e\d\c\d\b\d\b\d\a\a\b\d$ as follows.
$$\begin{array}{lcr}
O && E\\ \hline
\d\d\e|\c\e\d|\b\d\b\d\c\c\d|\a\b\d|\a && - \\
\d\d\e|\c\e\d|\b\d\dashline\b\d\c\c\d &^\cF& \a\a\b\d \\
\d\d\e|\c\e\d|\b\d\dashline\c\c\d &^\cP& \b\d\a\a\b\d \\
\d\d\e|\c\e\d|\c\dashline\c\d &^\cP& \b\d\b\d\a\a\b\d \\
\d\d\e|\c\e\d|\c &^\cS& \c\d\b\d\b\d\a\a\b\d \\
\d\dashline\d\e &^\cF& \c\c\e\d\c\d\b\d\b\d\a\a\b\d \\
\d &^\cS& \d\e\c\c\e\d\c\d\b\d\b\d\a\a\b\d \\
- && \d\e|\d|\c\c\e\d\c\d|\b\d|\b\d|\a\a\b\d
\end{array}$$
\end{example}

\subsection{The inverse map $\Psinv$}\label{sec:psinv}

In order to describe the inverse map $\Psinv:\We_n\to\Wo_n$, let us first introduce the notion of iterated standard factorization.

\begin{definition}\label{def:iterated} Let $\ell\in\L$ with $|\ell|\ge2$, and let $u\in\W\cup\{\infty\}$.
The {\em iterated standard factorization (ISF) of $\ell$ with respect to $u$} is the unique factorization
$$\ell=r_js_js_{j-1} \dots s_1,$$
for some $j\ge1$, defined as follows: 
\begin{enumerate}[label=(\alph*)]
\item for every $i\in[j]$, the word $s_i$ is the smallest proper suffix of $r_js_js_{j-1} \dots s_i$;
\item for every $i\in[j-1]$, the word $s_i$ has even length and it satisfies $s_i<u$;
\item the word $s_j$ has odd length or it satisfies $u\le s_{j}$.
\end{enumerate}
\end{definition}

The ISF can be found by starting from $\ell$ and repeatedly (for $i=1,2,\dots$) removing the 
smallest proper suffix  $s_i$ of the remaining word, until this suffix no longer satisfies condition (b), in which case it satisfies condition~(c). 
Recall that the smallest proper suffix is the one given by the standard factorization.
After removing each $s_i$, the remaining word belongs to $\L$ by Lemma~\ref{lem:standard}.
If $u=\infty$, this process stops when it finds a suffix $s_j$ of odd length. We will write $\ell=r_j\dashline s_j\dashline  s_{j-1}\dashline  \dots\dashline  s_1$ to indicate that this is an ISF.

\begin{example}\label{ex:adbccc}
The ISF of $\ell=\a\d\b\c\c\c$ with respect to $u=\c\c\d$ is $r_2\dashline s_2\dashline s_1 =\a\dashline\d\dashline\b\c\c\c$. Indeed, the smallest proper suffix of $\ell$ is $s_1=\b\c\c\c$, which has even length and satisfies $s_1<u$. After removing this suffix, the remaining word $\a\d$ has smallest proper suffix $s_2=\d$, which has odd length and satisfies $u\le s_2$.
\end{example}

\begin{example}\label{ex:ccedcd}
The ISF of $\ell=\c\c\e\d\c\d$ with respect to $u=\d\d\e$ is $r_2\dashline s_2\dashline s_1
=\c\dashline\c\e\d\dashline\c\d$. Now the smallest proper suffix of $\ell$ is $s_1=\c\d$, which has even length and satisfies $s_1<u$. The smallest proper suffix of $\c\c\e\d$ is $s_2=\c\e\d$, which has odd length, even though $s_2<u$.
\end{example}

\begin{example}\label{ex:adcdbcdcbcbc}
The ISF of $\ell=\a\d\c\d\b\c\d\c\b\c\b\c$ with respect to $u=\c$ is 
$\ell=r_4\dashline s_4\dashline s_3\dashline s_2\dashline s_1=\a\d\dashline\c\d\dashline\b\c\d\c\dashline\b\c\dashline\b\c$. Note that the word $s_4=\c\d$ has even length but it satisfies $u\le s_4$, so the factorization stops here. The ISF of $\ell$ with respect to $\infty$ is $\ell=r_5\dashline s_5\dashline s_4\dashline s_3\dashline s_2\dashline s_1=\a\dashline\d\dashline\c\d\dashline\b\c\d\c\dashline\b\c\dashline\b\c$.
\end{example}

We are now ready to define the map $\Psinv$.  

\begin{definition}[The map $\Psinv$]\label{def:psinv}
On input $w'\in\We_n$, initially set $(O',E')=(-,w')$. If $n$ is odd, remove the odd factor from $E'$ (which has length one) and put it in $O'$.
Iterate the following step as long as $E'$ is not empty:
\begin{itemize}
\item[$(\psinv)$] Let $O'=o'_1|o'_2|\dots|o'_h$ and $E'=e'_1|e'_2|\dots|e'_k$ be their Lyndon factorizations. If $O'$ is empty, let $o'_h=\infty$ by convention. Update $(O',E')$ to 
$$(O,E)=(O'e'_1,\,e'_2\dots e'_k) \hspace{46mm}\text{if $o'_h<e'_1$}; \hfill\quad\cSi$$
otherwise, let $e'_1=r_j\dashline s_j\dashline s_{j-1}\dashline \dots\dashline  s_1$ be the ISF of $e'_1$ with respect to $o'_h$ (as in Definition~\ref{def:iterated}), and update $(O',E')$ to 
$$(O,E)=\begin{cases}
(o'_1o'_2\dots o'_{h-1}r_js_jo'_h,\,s_{j-1}\dots s_1e'_2\dots e'_k) & \text{if $o'_h\le s_j$}, \hfill \quad\cPi\\
(O's_jr_j,\hfill s_{j-1}\dots s_1e'_2\dots e'_k) & \text{if $s_j<o'_h$
}. \hfill \quad\cFi
\end{cases}$$
\end{itemize}

Once $E'$ is empty, let $\Psinv(w')=O'$.
\end{definition}

We will see later that, at any point in the computation of $\Psinv$, the Lyndon factors of $O'$ are odd and distinct, and the Lyndon factors of $E'$ are even. In particular, $o'_h\neq e'_1$, so we have $e'_1<o'_h$ in cases~\cPi\ and~\cFi.
Let us give some examples of the map $\Psinv$.

\begin{example} Let $w'=\c\d\c\d\a\d\b\c\c\c\in\We_{10}$. The following table summarizes the computation of $\Psinv(w')=\d\a\d\c\c\d\b\c\c\c$, which reverses the steps from Example~\ref{ex:10}. When $e'_1<o'_h$, the dashed vertical bars indicate the ISF of $e'_1$ with respect to $o'_h$. For the third step, the ISF of $\a\d\b\c\c\c$ with respect to $\c\c\d$ was given in Example~\ref{ex:adbccc}.
$$\begin{array}{lcr}
  O' && E'\\ \hline
  -&& \c\dashline\d|\c\d|\a\d\b\c\c\c \\
  \d|\c &^\cFi& \c\d|\a\d\b\c\c\c \\
  \d|\c\c\d &^\cSi& \a\dashline\d\dashline\b\c\c\c \\
  \d|\a\d\c\c\d &^\cPi& \b\c\c\c \\
 \d|\a\d\c\c\d\b\c\c\c &^\cSi& -
\end{array}$$
\end{example}

\begin{example}
For $w'=\d\e\d\c\c\e\d\c\d\b\d\b\d\a\a\b\d\in\We_{17}$, here is the computation of $\Psinv(w')=\d\d\e\c\e\d\b\d\b\d\c\c\d\a\b\d\a$, which reverses the steps from Example~\ref{ex:17}.
 For the third step, the ISF of $\c\c\e\d\c\d$ with respect to $\d\d\e$ was given in Example~\ref{ex:ccedcd}.
$$\begin{array}{lcr}
O' && E'\\ \hline
- && \d\e|\d|\c\c\e\d\c\d|\b\d|\b\d|\a\a\b\d \\
\d && \d\e|\c\c\e\d\c\d|\b\d|\b\d|\a\a\b\d \\
\d\d\e &^\cSi& \c\dashline\c\e\d\dashline\c\d|\b\d|\b\d|\a\a\b\d \\
\d\d\e|\c\e\d|\c &^\cFi& \c\d|\b\d|\b\d|\a\a\b\d \\
\d\d\e|\c\e\d|\c\c\d &^\cSi& \b\dashline\d|\b\d|\a\a\b\d \\
\d\d\e|\c\e\d|\b\d\c\c\d &^\cPi& \b\dashline\d|\a\a\b\d \\
\d\d\e|\c\e\d|\b\d\b\d\c\c\d &^\cPi& \a\dashline\a\b\d \\
\d\d\e|\c\e\d|\b\d\b\d\c\c\d|\a\b\d|\a &^\cFi& - 
\end{array}$$
\end{example}

\begin{example}
For $w'=\c\a\d\c\d\b\c\d\c\b\c\b\c\in\We_{13}$, the computation of $\Psinv(w')=\a\d\c\d\c\b\c\d\c\b\c\b\c$ is given below.
 For the second step, the ISF of $\a\d\c\d\b\c\d\c\b\c\b\c$ with respect to $\c$ was given in Example~\ref{ex:adcdbcdcbcbc}.
$$\begin{array}{lcr}
O' && E'\\ \hline
- && \c|\a\d\c\d\b\c\d\c\b\c\b\c \\
\c && \a\d\dashline\c\d\dashline\b\c\d\c\dashline\b\c\dashline\b\c \\
\a\d\c\d\c &^\cPi& \b\c\d\c|\b\c|\b\c \\
\a\d\c\d\c\b\c\d\c &^\cSi& \b\c|\b\c \\
\a\d\c\d\c\b\c\d\c\b\c &^\cSi& \b\c \\
\a\d\c\d\c\b\c\d\c\b\c\b\c &^\cSi& -
\end{array}$$
\end{example}

Our goal is to show that $\Psi$ is a map from $\Wo_n$ to $\We_n$, that $\Psinv$ is a map from $\We_n$ to $\Wo_n$, and that they are inverses of each other, and hence both are bijections.
More specifically, we will show that each step of type \cS, \cP\ and \cF\ in Definition~\ref{def:psi} is reversed by a step of type \cSi, \cPi\ and \cFi\ in Definition~\ref{def:psinv}, respectively.
Before that, we state in the next subsection a few simple facts about Lyndon words that will be used in the proofs.

\subsection{Auxiliary results about Lyndon words}\label{sec:auxiliary}

\begin{lemma}[{\cite[Prop.~5.1.4]{Lothaire}}]\label{lem:concatenate} Let $\ell$ with $|\ell|\ge2$ and standard factorization $\ell=r\dashline s$. 
Then, for any $t\in\L$ such that $\ell<t$, the expression $\ell t$ is a standard factorization if and only if $t\le s$.
\end{lemma}

\begin{lemma}\label{lem:bigLyndon}
If $\ell\in\L$ and $u,v\in\W$ satisfy $u<\ell$ and $v\le \ell$, then $uv<\ell$.
\end{lemma}

\begin{proof}
If $u$ is not a prefix of $\ell$, then $u<\ell$ implies $uv<\ell$ by the properties of the lexicographic order. If $u$ is a prefix of $\ell$, we can write $\ell=us$ for some $s$ satisfying $\ell<s$, since $\ell\in\L$. Thus, $v\le \ell<s$, and so $uv<us=\ell$.
\end{proof}

It is shown in~\cite[Prop.~5.1.6]{Lothaire} that, in a Lyndon factorization $w=\ell_1|\ell_2|\dots|\ell_m$, the rightmost factor $\ell_m$ is the (lexicographically) smallest suffix of $w$, and so the Lyndon factors can be found by repeatedly removing the smallest suffix of the remaining word. The following lemma uses this idea to describe the Lyndon factors of $w$ in terms of the relative order of its suffixes (see also~\cite{HR} for a related approach). For $w=\w_1\w_2\dots \w_n$, we use the notation $w_{[p,q]}=\w_p\w_{p+1}\dots\w_q$ when $p\le q$.

\begin{lemma}\label{lem:LRmin}
Let $w\in\W_n$. The starting positions of the Lyndon factors of $w$ are the left-to-right minima of the sequence of suffixes
\begin{equation}\label{eq:suffixes} w_{[1,n]},w_{[2,n]},\dots,w_{[n,n]},\end{equation} 
that is, the indices $q$ such that $w_{[q,n]}<w_{[p,n]}$ for all $p<q$.
\end{lemma}

\begin{proof}
Let $1=q_1<q_2<\dots<q_m$ be the left-to-right minima of the sequence~\eqref{eq:suffixes}, and use the convention $q_{m+1}=n+1$. 
The factors of $w$ that start at these left-to-right minima are $\ell_i=w_{[q_i,q_{i+1}-1]}$, for $i\in[m]$. We want to show that $w=\ell_1|\ell_2|\dots|\ell_m$ is the Lyndon factorization of~$w$.

To prove that $\ell_i\in\L$, suppose for contradiction that $\ell_i$ had a smaller suffix, say $w_{[q,q_{i+1}-1]}<\ell_i$ for some $q_i<q<q_{i+1}$. Then $w_{[q,n]}<w_{[q_i,n]}$, contradicting that the sequence~\eqref{eq:suffixes} has no left-to-right minima strictly between $q_i$ and $q_{i+1}$.

To prove that $\ell_i\ge\ell_{i+1}$ for all $i\in[m-1]$, suppose for contradiction that $\ell_i<\ell_{i+1}$. Then $\ell_i\ell_{i+1}<\ell_{i+1}$ by Lemma~\ref{lem:bigLyndon}, and so $w_{[q_i,n]}=\ell_i\ell_{i+1}\dots\ell_m<\ell_{i+1}\dots\ell_m=w_{[q_{i+1},n]}$, contradicting that $q_{i+1}$ is a left-to-right minimum.
\end{proof}

\begin{lemma}\label{lem:iterated}
Let $\ell=r_j\dashline s_j\dashline s_{j-1}\dashline  \dots \dashline s_1$ be the ISF of $\ell$ with respect to any $u$, as in Definition~\ref{def:iterated}. Then  $$r_j<\ell<s_1\le s_2\le\dots\le s_{j-1}\le s_j.$$
\end{lemma}

\begin{proof}
Since $s_1$ is the longest proper suffix of $\ell$ that is a Lyndon word, we have $r_j\le r_js_js_{j-1} \dots s_2<\ell<s_1$ by Lemma~\ref{lem:standard}. The fact that $s_1\le s_2\le\dots\le s_{j}$ follows from Lemma~\ref{lem:concatenate}, since each $s_i$ is the suffix of a standard factorization.
\end{proof}

\begin{lemma}\label{lem:no-overlap}
Let $r\in\L$ with $|r|\ge2$ and standard factorization $r=r'\dashline s'$, and let $u$ be a nonempty word such that $u\le s'$ and $ru\in\L$.
Then the smallest proper suffix of $ru$ is contained in $u$.
\end{lemma}

\begin{proof}
Let $t$ be the smallest proper suffix of $ru$, and suppose for contradiction that it overlaps with~$r$.
Since the smallest proper suffix of $r$ is $s'$, we must have $t=s'u$. But $u\le s'$ implies $u<s'u=t$, contradicting that $t$ is the smallest proper suffix of $ru$.
\end{proof}

\subsection{Properties of $\Psi$ and $\Psinv$}\label{sec:properties}

The next lemma describes some properties of the pairs of words $(O,E)$ that appear in the computation of $\Psi$. We write $(O,E)\overset{\psi}{\to}(O',E')$ to indicate that $(O',E')$ is obtained by applying one step $(\psi)$ to $(O,E)$, and we use the same notation as in Definition~\ref{def:psi}.

\begin{lemma}\label{lem:O'E'}
Let $(O,E)\overset{\psi}{\to}(O',E')$ be a step in the computation of $\Psi$ on some input $w\in\Wo_n$. Denote the Lyndon factorization of $O'$ by $o'_1|o'_2|\dots|o'_h$, and use the convention $o'_i=\infty$ for $i\le 0$.
\begin{enumerate}[label = (\roman*)]
\item\label{part:factO'} The word $O'$ has odd and distinct Lyndon factors, given, after each type of step, by
$$O'=o'_1|o'_2|\dots|o'_h=
\begin{cases} o_1|o_2|\dots |o_{m-1}|r & \quad\cS,\\
o_1|o_2|\dots|o_{m-1}|s & \quad\cP,\\
o_1|o_2|\dots|o_{m-2} & \quad\cF.
\end{cases}$$
\item\label{part:o-1} 
$E'<o'_{h-1}$.
\item\label{part:F} After a step of type \cF, $E'<o'_h$. 
\item\label{part:P} After a step of type \cP, $E'<o'_h$ and $E<o'_h$. 
\item\label{part:S} After a step of type \cS, $s$ is the leftmost Lyndon factor of~$E'$, and $o'_h<s\le E'$.
\item\label{part:even} All the Lyndon factors of $E'$ are even. Additionally, if $E'=\ell E$, then $\ell$ is contained in the leftmost Lyndon factor of $E'$.
\end{enumerate}
\end{lemma}

\begin{proof}
We prove part~\ref{part:factO'} by induction on the number of steps taken in the computation of $\Psi$ up to the current step, noting that the input word $w$ has odd and distinct Lyndon factors. Suppose that the Lyndon factorization $O=o_1|o_2|\dots|o_{m}$ consists of distinct odd factors  $o_1>o_2>\dots>o_m$. Then the factors of $O'$ in case~\cF\ are $o_1>o_2>\dots>o_{m-2}$, which are odd and distinct as well. In cases~\cS\ and~\cP, we have $r<s<o_{m-1}$ (by Lemma~\ref{lem:standard} and the fact that $o_m$ is splittable), and so $o_1>o_2>\dots>o_{m-1}>s>r$. This proves that the Lyndon factorizations are as stated in both cases, and that the factors are odd and distinct, since $r$ is odd in case~\cS, and $s$ is odd in case~\cP.

Part~\ref{part:o-1} is also proved by induction on the number of steps taken. First note that the initial pair of words $(w,-)$ trivially satisfies the condition. Suppose now that $E<o_{m-1}$. We want to show that $E'<o'_{h-1}$.

In cases \cS\ and \cP, recall that $r<s<o_{m-1}$. By Lemma~\ref{lem:bigLyndon} applied to $o_{m-1}\in\L$, it follows that $E'=sE<o_{m-1}$ in case \cS, and $E'=rE<o_{m-1}$ in case \cP. In both cases, $o'_{h-1}=o_{m-1}$ by part~\ref{part:factO'}, proving that $E'<o'_{h-1}$.

In case \cF, the inequalities $o_{m}<o_{m-1}$ and $E<o_{m-1}$ imply that $E'=o_{m}o_{m-1}E<o_{m-1}$, again by Lemma~\ref{lem:bigLyndon}. By part~\ref{part:factO'}, $o'_h=o_{m-2}$, and so $E'<o_{m-1}<o_{m-2}=o'_h\le o'_{h-1}$, proving not only part~\ref{part:o-1} in this case, but also part~\ref{part:F}.

To prove part~\ref{part:P}, we again use induction on the number of steps taken, but this time, if $(O,E)$ follows a step of type~\cS, we backtrack to the most recent pair $(\hat{O},\hat{E})$ in the computation of $\Psi$ that did not follow a step of type~\cS. 
We write $$(\hat{O},\hat{E})\overset{\cS^{j-1}}{\longrightarrow}(O,E)\overset{\cP}{\rightarrow}(O',E')$$
to indicate that $(O,E)$ is obtained from $(\hat{O},\hat{E})$ by applying a sequence of $j-1\ge0$ steps of type~\cS, and $(O',E')$ is obtained from $(O,E)$ by a step of type~\cP.
Letting $\hat{O}=\hat{o}_1|\dots|\hat{o}_p$, we have $\hat{E}<\hat{o}_p$. Indeed, this is trivial if $\hat{E}=-$, and otherwise it is a consequence of part~\ref{part:F} or the induction hypothesis on part~\ref{part:P}, depending on whether $(\hat{O},\hat{E})$ followed a step of type~\cF\ or~\cP, respectively.
Since the algorithm applies $j-1$ steps of type~\cS\ followed by a step of type~\cP\ to $(\hat{O},\hat{E})$, the ISF of $\hat{o}_p$ with respect to $\hat{o}_{p-1}$ can be written as $\hat{o}_p=r_{j}\dashline s_{j}\dashline s_{j-1}\dashline \dots\dashline  s_1$, where $s_i$ is even for $i\in[j-1]$, and $s_j$ is odd. Then $$(O,E)=(\hat{o}_1\dots\hat{o}_{p-1}r_js_j,s_{j-1}\dots s_1\hat{E}) \quad\text{and}\quad (O',E')=(\hat{o}_1\dots\hat{o}_{p-1}s_j,r_jE).$$ By Lemma~\ref{lem:iterated}, we have $r_j<s_1\le s_2\le\dots\le s_{j-1}< s_j$, where the last inequality is strict because $s_{j-1}$ and $s_j$ have different length. Lemma~\ref{lem:bigLyndon} now implies that $\hat{o}_p=r_{j}s_{j}s_{j-1}\dots s_1<s_j$, and so  $\hat{E}<\hat{o}_p<s_j$. Applying Lemma~\ref{lem:bigLyndon} again, we deduce that $E=s_{j-1}\dots s_1\hat{E}<s_j$, and that $E'=r_jE<s_j$, which proves part~\ref{part:P} since $o'_h=s_j$.

Next we prove part~\ref{part:S}. After a step of type~\cS, it is clear by part~\ref{part:factO'} that $o'_h=r<s\le sE=E'$. 
To show that $s$ is the leftmost Lyndon factor of $E'$, we again use induction on the number of steps. As in the previous paragraph, 
we backtrack to the most recent pair $(\hat{O},\hat{E})$ that did not follow a step of type~\cS, and write
$$(\hat{O},\hat{E})\overset{\cS^{j-1}}{\longrightarrow}(O,E)\overset{\cS}{\rightarrow}(O',E')$$
for some $j\ge1$. Letting $\hat{O}=\hat{o}_1|\dots|\hat{o}_p$, we have $\hat{E}<\hat{o}_p$ by parts~\ref{part:F} and~\ref{part:P}. 
Since the algorithm applies $j$ steps of type~\cS\ to $(\hat{O},\hat{E})$, we have a factorization $\hat{o}_p=r_{j}\dashline s_{j}\dashline s_{j-1}\dashline \dots \dashline s_1$ where each $s_i$ is the smallest proper suffix of the remaining word, and so 
$\hat{o}_p<s_1\le s_2\le \dots \le s_j$ by Lemma~\ref{lem:concatenate}.
Then $$(O,E)=(\hat{o}_1\dots\hat{o}_{p-1}r_js_j,s_{j-1}\dots s_1\hat{E}) \quad\text{and}\quad (O',E')=(\hat{o}_1\dots\hat{o}_{p-1}r_j,s_jE).$$
Since $\hat{E}<\hat{o}_p<s_1\le\dots\le s_j$, we deduce that $s=s_j$ is the leftmost Lyndon factor of $E'$.

Part~\ref{part:even} can also be proved by induction on the number of steps taken, noting that, initially, the empty word trivially satisfies that all its Lyndon factors are even. By construction, $E'=\ell E$ for some even $\ell\in\L$. Specifically, $\ell=s$ in case~\cS, $\ell=r$ in case~\cP, and $\ell=o_m o_{m-1}$ in case~\cF, which are Lyndon words by Lemma~\ref{lem:standard}. By induction hypothesis, all the Lyndon factors of $E$ are even. 
By Lemma~\ref{lem:LRmin}, the starting positions of the Lyndon factors are the left-to-right minima of the sequence of suffixes.
Since every left-to-right minimum of the sequence of suffixes of $\ell E$, other than the one at the very left, must also be a left-to-right minimum of the sequence of suffixes of $E$ (with the obvious shift by $|\ell|$), we deduce that all the Lyndon factors of $E'$ are even as well, and that $\ell$ is contained in the leftmost Lyndon factor of $E'$.
\end{proof}

The next lemma shows that each step of type \cS, \cP\ and \cF\ in the computation of $\Psi$ is reversed by a step of type \cSi, \cPi\ and \cFi\ in the computation of $\Psinv$, respectively.

\begin{lemma}\label{lem:psinv_psi}
Let $(O,E)\overset{\psi}{\to}(O',E')$ be a step in the computation of $\Psi$ on some input $w\in\Wo_n$. 
Then $(O,E)$ is obtained from $(O',E')$ by applying a step $(\psinv)$ (as in Definition~\ref{def:psinv}).
\end{lemma}

\begin{proof}
Let $e'_1$ be the leftmost Lyndon factor of $E'$, which is even by Lemma~\ref{lem:O'E'}\ref{part:even}. Suppose first that the step $(O,E)\overset{\psi}{\to}(O',E')$ was of type~\cS, so $(O',E')=(o_1o_2\dots o_{m-1}r,\,sE)$. In this case,
$r=o'_h<e'_1=s$ by Lemma~\ref{lem:O'E'}\ref{part:S}, and so Definition~\ref{def:psinv} applies a step of type \cSi\ to $(O',E')$, producing the pair $(o_1o_2\dots o_{m-1}rs,\,E)=(O,E)$.

If the step $(O,E)\overset{\psi}{\to}(O',E')$ was of type~\cP\ or~\cF, we know by Lemma~\ref{lem:O'E'}\ref{part:F}\ref{part:P} that $e'_1\le E'<o'_h$.
When $(O',E')$ satisfies this condition, Definition~\ref{def:psinv} computes the ISF $e'_1=r_j\dashline s_j\dashline s_{j-1}\dashline \dots \dashline s_1$ with respect to $o'_h$, and compares $o'_h$ with $s_j$ to determine whether to apply a step of type \cPi\ or \cFi.
\smallskip

Suppose first that the step $(O,E)\overset{\psi}{\to}(O',E')$ was of type \cP, so
$(O',E')=(o_1o_2\dots o_{m-1}s,rE)$, and $E<o'_h=s$ by Lemma~\ref{lem:O'E'}\ref{part:factO'}\ref{part:P}. By Lemma~\ref{lem:O'E'}\ref{part:even}, $r$ is contained in $e'_1$, so we can write $e'_1=rv$ for some (possibly empty) word $v$, which is a prefix of $E$, and so $v\le E<s$. 
Letting $r=r'\dashline s'$ be the standard factorization of $r$, we have $s\le s'$ by Lemma~\ref{lem:concatenate}, since $r\dashline s$ is a standard factorization. 

Recall that the factors  $s_i$ in the ISF of $e'_1=rv$ with respect to $o'_h=s$ are obtained by repeatedly removing the smallest proper suffix. As long as a non-empty portion of $v$, say $u_i$, remains in this process, Lemma~\ref{lem:no-overlap} (using the fact that $u_i\le v<s'$) implies that the smallest proper suffix $s_i$ of $ru_i$ must be contained in $u_i$.
In this case, since all the Lyndon factors of $E$ are even by Lemma~\ref{lem:O'E'}\ref{part:even}, all the left-to-right minima of its sequence of suffixes must have odd indices by Lemma~\ref{lem:LRmin}, and hence the same is true for the left-to-right minima of the sequence of suffixes of any prefix $u_i$ of $E$. It follows that $s_i$ has even length, and since $s_i\le u_i\le v<s$ (otherwise $s_i$ would not have been the smallest proper suffix), the process of computing the ISF of $e'_1$ with respect to $s$ continues until all of $v$ has been removed. 

Let $i\ge0$ be the largest index such that $s_i$ is contained in $v$ as in the above paragraph, so that $v=s_is_{i-1}\dots s_1$.
Since the standard factorization $r=r'\dashline s'$ satisfies $s\le s'$, the computation of the ISF of $e'_1$ with respect to $s$ stops here, that is, we have $j=i+1$, $r_j=r'$ and $s_j=s'$. At this point, since $o'_h=s\le s'=s_j$, Definition~\ref{def:psinv} applies a step of type~\cPi\ to $(O',E')$, recovering $(o_1o_2\dots o_{m-1}rs,E)=(O,E)$.
\smallskip

Finally, suppose that the step $(O,E)\overset{\psi}{\to}(O',E')$ was of type \cF. In this case, $(O',E')=(o_1o_2\dots o_{m-2},o_mo_{m-1}E)$, and the argument is similar to the previous case. 
By Lemma~\ref{lem:O'E'}\ref{part:even}, the word $o_mo_{m-1}$ is contained inf $e'_1$, so we can write $e'_1=o_mo_{m-1}v$ for some (possibly empty) word $v$, where $v\le E<o_{m-1}<o_{m-2}=o'_h$ by Lemma~\ref{lem:O'E'}\ref{part:o-1} applied to~$E$. 

Let us first show that $o_m\dashline o_{m-1}$ is a standard factorization. Since the step $(O,E)\overset{\psi}{\to}(O',E')$ was of type \cF, $o_m$ is not splittable. This means that either $|o_m|=1$, in which case $o_{m-1}$ is clearly the longest proper suffix of $o_mo_{m-1}$ that belongs to $\L$, or its standard factorization $o_m=r\dashline s$ satisfies $o_{m-1}\le s$, in which case Lemma~\ref{lem:concatenate} implies that $o_m\dashline o_{m-1}$ is a standard factorization.

The computation of the ISF of $e'_1$ with respect to $o'_h=o_{m-2}$ repeatedly removes the smallest proper suffix $s_i$. 
As long as a non-empty portion $u_i$ of $v$ remains, Lemma~\ref{lem:no-overlap} (using the fact that $u_i\le v<o_{m-1}$) implies that $s_i$ must be contained in $u_i$. As in the previous case, $s_i$ must have even length, and since $s_i\le u_i\le v<o_{m-2}$, the computation of the ISF continues until all of $v$ has been removed. 

Let $i\ge0$ be the largest index such that $s_i$ is contained in $v$, so that $v=s_is_{i-1}\dots s_1$.
Since $o_m\dashline o_{m-1}$ is a standard factorization and $o_{m-1}$ is odd, the computation of the ISF of $e'_1$ stops here, that is, we have $j=i+1$, $r_j=o_m$ and $s_j=o_{m-1}$. At this point, since $s_j=o_{m-1}<o_{m-2}=o'_h$, Definition~\ref{def:psinv} applies a step of type~\cFi\ to $(O',E')$, producing $(o_1o_2\dots o_{m-2}o_{m-1}o_m,E)=(O,E)$.
\end{proof}

The next lemma describes some properties of the pairs $(O',E')$ in the computation of $\Psinv$. We write by $(O',E')\overset{\psinv}{\to}(O,E)$ to indicate that $(O,E)$ is obtained by applying one step $(\psinv)$ to $(O',E')$, and we use the same notation as in Definition~\ref{def:psinv}.

\begin{lemma}\label{lem:OE}
Let $(O',E')\overset{\psinv}{\to}(O,E)$ be a step in the computation of $\Psinv$ on some input $w'\in\We_n$. Denote the Lyndon factorization of $O$ by $o_1|o_2|\dots|o_m$, and use the convention $o_i=o'_i=\infty$ for $i\le0$.
\begin{enumerate}[label = (\roman*)]
\item\label{part:eveninv} All the Lyndon factors of $E$ are even.
\item\label{part:factO} The word $O$ has odd and distinct Lyndon factors, given, after each type of step, by
$$O=o_1|o_2|\dots|o_m=
\begin{cases} o'_1|o'_2|\dots |o'_{h-1}|o'_he'_1 & \quad\cSi,\\
o'_1|o'_2|\dots|o'_{h-1}|r_js_jo'_h & \quad\cPi,\\
o'_1|o'_2|\dots|o'_h|s_j|r_j & \quad\cFi.
\end{cases}$$
\item\label{part:o'-1} $E<o_{m-1}$.
\item\label{part:e1} If $|o_m|\ge2$ and its standard factorization is $o_m=r\dashline s$, then the leftmost Lyndon factor $e_1$ of $E$ satisfies $e_1\le s$.
\item\label{part:Si} After a step of type \cSi, the standard factorization of $o_m$ is $o'_h\dashline e'_1$.
\item\label{part:Pi} After a step of type \cPi, the standard factorization of $o_m$ is $r_js_j\dashline o'_h$.
\end{enumerate}
\end{lemma}

\begin{proof}
Part~\ref{part:eveninv} is proved by induction on the number of steps taken in the computation of $\Psinv$ up to the current step. Initially, the word $w'$, after its factor of length one has been removed if $n$ is odd, has only even Lyndon factors. Suppose that all the Lyndon factors of $E'=e'_1|e'_2|\dots|e'_k$ are even, and let us show that the same holds for $E$. This is clear after a step of type~\cSi, since $E'=e'_2|\dots|e'_k$.
After a step of type~\cPi\ or~\cFi, recall that the ISF $e'_1=r_j\dashline s_j\dashline s_{j-1}\dashline\dots\dashline  s_1$ with respect to $o'_h$ satisfies
\begin{equation}\label{eq:es} e'_1<s_1\le\dots\le s_{j-1}\le s_j \end{equation}
by Lemma~\ref{lem:iterated}, where $s_i$ is an even Lyndon word for $i\in[j-1]$, and that $E=s_{j-1}s_{j-2}\dots s_1e'_2\dots e'_k$. It follows that the Lyndon factors of $E$ are $s_{j-1}\ge s_{j-2}\ge \dots \ge s_1> e'_2\ge \dots\ge e'_k$, all of which are even.

We prove parts~\ref{part:factO} and~\ref{part:o'-1} together, also by induction on the number of steps taken. At the start of the algorithm, the initial word $O'$ is either empty or has length $1$, so it trivially has odd and distinct factors, and it satisfies $E'<o'_{h-1}=\infty$. Suppose that, in the current step, the Lyndon factors of $O'$ are odd and distinct (i.e., $o'_1>o'_2>\dots>o'_h$), and that $E'<o'_{h-1}$. We want to show that $O$ has Lyndon factors as given by part~\ref{part:factO}, and that $E<o_{m-1}$.

After a step of type~\cSi, since $o'_h<e'_1$, we have $o'_he'_1\in\L$ by Lemma~\ref{lem:standard}, and $o'_he'_1<o'_{h-1}$ by Lemma~\ref{lem:bigLyndon}, using that $e'_1\le E'<o'_{h-1}$. We conclude that $O$ has distinct factors 
$o'_1>o'_2>\dots>o'_{h-1}>o'_he'_1$, and since $e'_1$ is even by part~\ref{part:eveninv} applied to $E'$, the new factor $o'_he'_1$ is odd. Additionally, we have $E=e'_2\dots e'_k<e'_1e'_2\dots e'_k=E'<o'_{h-1}=o_{m-1}$.

After a step of type~\cPi, since $r_js_j\le e'_1<o'_h$, we have
$r_js_jo'_h\in\L$ by Lemma~\ref{lem:standard}, and $r_js_jo'_h<o'_h<o'_{h-1}$ by Lemma~\ref{lem:bigLyndon}.
This shows that $O$ has distinct Lyndon factors $o'_1>o'_2>\dots>o'_{h-1}>r_js_jo'_h$, and since $r_js_j$ is even, the new factor $r_js_jo'_h$ is odd.
Additionally, by Defintion~\ref{def:iterated}, $s_{j-1}<o'_h<o'_{h-1}=o_{m-1}$. Using equation~\eqref{eq:es} and Lemma~\ref{lem:bigLyndon}, together with the fact that $e'_1\ge e'_2\ge\dots\ge e'_k$, we deduce that $E<o_{m-1}$.

After a step of type~\cFi, we have $r_j<s_j<o'_h$, and so the ISF stopped because $s_j$ (and hence $r_j$) is odd. It follows
that $O$ has odd and distinct Lyndon factors $o'_1>o'_2>\dots>o'_h>s_j>r_j$.
Additionally, since $s_{j-1}<s_j=o_{m-1}$ (using that $s_{j-1}$ and $s_j$ have different length), again equation~\eqref{eq:es} and Lemma~\ref{lem:bigLyndon} imply that $E<o_{m-1}$. \smallskip

Let us now prove part~\ref{part:Pi}. After a step of type~\cPi, we have $o_m=r_js_jo'_h$ by part~\ref{part:factO}, and $o'_h\le s_j$ by definition. Since $r_j\dashline s_j$ is a standard factorization, Lemma~\ref{lem:no-overlap} implies that the smallest proper suffix of $r_js_jo'_h$ is contained in $o'_h$. But since $o'_h\in\L$, the smallest proper suffix is $o'_h$ itself.
\smallskip

Finally, we prove parts~\ref{part:e1} and~\ref{part:Si} by induction on the number of steps taken. Suppose that, if $|o'_h|\ge2$, its standard factorization $o'_h=r'\dashline s'$ satisfies $e'_1\le s'$. Note that, if $O'$ is empty, then Definition~\ref{def:psinv} applies a step of type~\cFi, so we can assume that $O'$ is not empty in the next two paragraphs.

After a step of type~\cSi, we have $o_m=o'_he'_1$ by part~\ref{part:factO}. 
Let us first show that $o'_h\dashline e'_1$ is a standard factorization. If $|o'_h|=1$, this is clear because $e'_1$ is a Lyndon word; if $|o'_h|\ge2$, this follows from Lemma~\ref{lem:concatenate} using the induction hypothesis $e'_1\le s'$. 
We conclude that $e_1=e'_2\le e'_1=s$ in this case.

After a step of type~\cPi, we know by part~\ref{part:Pi} that $o_m=r_js_j\dashline o'_h$ is a standard factorization. 
If $j=1$, we conclude that $e_1=e'_2\le e'_1<o'_h=s$. If $j\ge2$, then the leftmost Lyndon factor of $E$ is $e_1=s_{j-1}$, by equation~\eqref{eq:es}, and so $e_1=s_{j-1}<o'_h=s$.

After a step of type~\cFi, we have $o_m=r_j$ by part~\ref{part:factO}. If $|o_m|=1$ there is nothing to prove, so suppose $|o_m|\ge2$. Since $r_j\dashline s_j$ is a standard factorization, the standard factorization $r_j=r\dashline s$ satisfies $s_j\le s$ by Lemma~\ref{lem:concatenate}. If $j=1$, we have $e_1=e'_2\le e'_1<s_j\le s$. If $j\ge2$, we have $e_1=s_{j-1}<s_j\le s$.
\end{proof}

The next lemma shows that each step of type \cSi, \cPi\ and \cFi\ in the computation of $\Psinv$ is reversed by a step of type \cS, \cP\ and \cF\ in the computation of $\Psi$, respectively.

\begin{lemma}\label{lem:psi_psinv}
Let $(O',E')\overset{\psinv}{\to}(O,E)$ be a step in the computation of $\Psinv$ on some input $w'\in\We_n$. 
Then $(O',E')$ is obtained from $(O,E)$ by applying a step $(\psi)$ (as in Definition~\ref{def:psi}).
\end{lemma}

\begin{proof}
If the step $(O',E')\overset{\psinv}{\to}(O,E)$ was of type~\cSi, we know by Lemma~\ref{lem:OE}\ref{part:Si} that the standard factorization of the rightmost Lyndon factor of $O$ is $o_m=o'_h\dashline e'_1$. By Lemma~\ref{lem:OE}\ref{part:o'-1} applied to $E'$, we have $e'_1\le E'<o'_{h-1}=o_{m-1}$. Thus, $o_m$ is splittable, and since $o'_h$ is odd, Definition~\ref{def:psi} applies a step of type~\cS\ to $(O,E)=(O'e'_1,e'_2\dots e'_k)$, recovering $(O',E')$.

If the step $(O',E')\overset{\psinv}{\to}(O,E)$ was of type~\cPi, we know by Lemma~\ref{lem:OE}\ref{part:Pi} that the standard factorization of the rightmost Lyndon factor of $O$ is $o_m=r_js_j\dashline o'_h$. Since $o'_h<o'_{h-1}=o_{m-1}$, again $o_m$ is splittable, and since $r_js_j$ is even, Definition~\ref{def:psi} applies a step of type~\cP\ to $(O,E)=(o'_1o'_2\dots o'_{h-1}r_js_jo'_h,s_{j-1}\dots s_2s_1e'_2\dots e'_k)$, moving $r_js_j$ to the beginning of $E$ and recovering $(O',E')$.

If the step $(O',E')\overset{\psinv}{\to}(O,E)$ was of type~\cFi, then $o_m=r_j$ and $o_{m-1}=s_j$  by Lemma~\ref{lem:OE}\ref{part:factO}. As in the proof of Lemma~\ref{lem:OE}\ref{part:e1}, either $|o_m|=1$, or its standard factorization $o_m=r\dashline s$ satisfies $o_{m-1}=s_j\le s$. This means that $o_m$ is not splittable, so Definition~\ref{def:psi} applies a step of type~\cF\ to $(O,E)=(O's_jr_j,s_{j-1}\dots s_2s_1e'_2\dots e'_k)$, placing 
$r_j$ and $s_j$ (in reverse order) at the beginning of $E$ and recovering $(O',E')$.
\end{proof}

\begin{proof}[Proof of Theorem~\ref{thm:bijection}]
Let us show that $\Psi$ is a bijection from $\Wo_n$ to $\We_n$ with inverse $\Psinv$.
Given $w\in\Wo_n$, each iteration of the step $(\psi)$ in the computation of $\Psi(w)$ from Definition~\ref{def:psi} decreases the length of the word $O$ by a positive even amount. 
Suppose first that $n$ is even, so that this process ends when $O$ is empty. By Lemma~\ref{lem:O'E'}\ref{part:even}, all the Lyndon factors of $\Psi(w)$ are even, and so $\Psi(w)\in\We_n$. 
Additionally, by Lemma~\ref{lem:psinv_psi}, each step $(\psi)$ in this computation is reversed by a step $(\psinv)$ in the computation of $\Psinv$ on input $\Psi(w)$, as in Definition~\ref{def:psinv}.
It follows that $\Psinv(\Psi(w))=w$.
Similarly, for any given $w'\in\We_n$, all the Lyndon factors of $\Psinv(w')$ are odd and distinct by Lemma~\ref{lem:OE}\ref{part:factO}, so $\Psinv(w')\in\Wo_n$. Additionally, by Lemma~\ref{lem:psi_psinv}, each step $(\psinv)$ in the computation of $\Psinv(w')$ is reversed by a step $(\psi)$ in the computation of $\Psi$ on input $\Psinv(w')$, so $\Psi(\Psinv(w'))=w'$.

In the case when $n$ is odd, the argument is similar, except that now both maps $\Psi$ and $\Psinv$ apply an additional step. Given $w\in\Wo_n$, the computation of $\Psi(w)$ eventually reaches the case $|O|=1$, at which point the remaining letter is inserted in $E$ as a new Lyndon factor of length one, so $\Psi(w)\in\We_n$ in this case as well. 
On the other hand, in the computation of $\Psinv(w')$ for $w'\in\We_n$, the initialization step moves the (unique) Lyndon factor of $E'$ of length one to $O'$, to ensure that only even factors remain in $E'$. These additional steps in Definitions~\ref{def:psi} and~\ref{def:psinv} are clearly inverses of each other.
The remaining steps proceed as in the case of even~$n$, so we conclude that $\Psinv(\Psi(w))=w$ for all $w\in\Wo_n$, and $\Psi(\Psinv(w'))=w'$ for all  $w'\in\We_n$.

Finally, the bijections $\Psi$ and $\Psinv$ are clearly weight-preserving, since all their steps move subwords between $O$ and $E$ without changing the total weight $\wt(O)\wt(E)$.
\end{proof}

\begin{proof}[Proof of Theorem~\ref{thm:main}]
Combining Theorem~\ref{thm:bijection} with Propositions~\ref{prop:Phi} and~\ref{prop:Xi}, and identifying $\Mo_n$ with $\Wo_n$ and $\Me_n$ with $\We_n$ as described in Section~\ref{sec:GF}, it follows that the map $f_S=\Phi_S^{-1}\circ\Psi\circ\Xi_S$ is a bijection from $\{\pi\in\So_n:\Asc(\pi)\subseteq S\}$ to $\{\pi\in\Se_n: \Des(\pi)\subseteq S\}$.
\end{proof}

Below are two examples of the computation of $f_S$ as a composition of the three bijections.

\begin{example}
Let $n=17$ and $S=\{2,5,8,15\}$, and let 
\begin{align*}\pi&=3\ 2\ 15\ 13\ 11\ 16\ 14\ 7\ 17\ 9\ 8\ 6\ 5\ 4\ 1\ 12\ 10\\
&=(9,17,10)(6,16,12)(4,13,5,11,8,7,14)(2)(1,3,15)\in\So_{17},\end{align*}
which has $\Asc(\pi)=S$, so it belongs to the left-hand side of equation~\eqref{eq:main}. To compute $\Xi_S(\pi)$, we make the replacements $\{1,2\}\mapsto \a$, $\{3,4,5\}\mapsto\b$, $\{6,7,8\}\mapsto\c$, $\{9,10,11,12,13,14,15\}\mapsto\d$ and $\{16,17\}\mapsto\e$ in the cycle form of $\pi$. The resulting multiset of necklaces
$$\Xi_S(\pi)=(\d,\e,\d)(\c,\e,\d)(\b,\d,\b,\d,\c,\c,\d)(\a)(\a,\b,\d)$$
can be identified with the word
$w=\d\d\e|\c\e\d|\b\d\b\d\c\c\d|\a\b\d|\a$.
As in Example~\ref{ex:17}, we have $$\Psi(w)=\d\e|\d|\c\c\e\d\c\d|\b\d|\b\d|\a\a\b\d,$$
which corresponds to the multiset of necklaces
$$(\d,\e)(\d)(\c,\c,\e,\d,\c,\d)(\b,\d)(\b,\d)(\a,\a,\b,\d).$$
Next, to apply $\Phi_S^{-1}$ to this multiset of necklaces, we replace the elements by their periodic labels
\begin{multline*}
(\d\e\d\dots,\e\d\e\dots)(\d\d\d\dots)(\c\c\e\dots,\c\e\d\dots,\e\d\c\dots,\d\c\d\dots,\c\d\c\dots,\d\c\c\dots)\\
(\b\d\b\dots,\d\b\d\dots)(\b\d\b\dots,\d\b\d\dots)(\a\a\b\dots,\a\b\d\dots,\b\d\a\dots,\d\a\a\dots).\end{multline*}
and order them lexicographically, breaking ties consistently, to get the permutation
\begin{align*} f_S(\pi)=\Phi_S^{-1}(\Psi(\Xi_S(\pi)))&=(15,17)(14)(6,8,16,13,7,12)(5,11)(4,10)(1,2,3,9)\\
&=2\ 3\ 9\ 10\ 11\ 8\ 12\ 16\ 1\ 4\ 5\ 6\ 7\ 14\ 17\ 1\ 15 \in\Se_{17},\end{align*}
which has $\Des(f_S(\pi))=\{5,8,15\}\subseteq S$.
\end{example}

\begin{example}
Now let $n=8$ and $S=\{4,7\}$, and let $\pi=75218634=(6)(1,7,3,2,5,8,4)\in\So_{8}$,
which has $\Asc(\pi)=S$. Making the replacements $\{1,2,3,4\}\mapsto \a$, $\{5,6,7\}\mapsto\b$ and $\{8\}\mapsto\c$, in the cycle form of $\pi$, we get the multiset of necklaces $\Xi_S(\pi)=(\b)(\a,\b,\a,\a,\b,\c,\a)$, which
can be identified with the word
$w=\b|\a\a\b\a\a\b\c$.
Next we compute $\Psi(w)=\a\b|\a\b|\a\a\b\c$, as shown below.
$$\begin{array}{lcr}
O && E\\ \hline
\b|\a\a\b\dashline\a\a\b\c && - \\
\b|\a\dashline\a\b &^\cS& \a\a\b\c \\
\b|\a &^\cS& \a\b\a\a\b\c \\
- &^\cF& \a\b\a\b\a\a\b\c
\end{array}$$
The word $\Psi(w)$ corresponds  to the multiset of necklaces $(\a,\b)(\a,\b)(\a,\a,\b,\c)$.
Finally, we apply $\Phi_S^{-1}$ to this multiset, as in Example~\ref{ex:Phi}, to get the permutation
$$f_S(\pi)=\Phi_S^{-1}(\Psi(\Xi_S(\pi)))=(3,6)(2,5)(1,4,7,8)=45672381\in\Se_8,$$
which has $\Des(f_S(\pi))=\{4,7\}= S$.
\end{example}

\section{Final remarks}

The maps $\Phi$ and $\Xi$ were originally defined in~\cite{GR,GRR} as bijections from words to multisets of necklaces. Specifically, instead of starting from the permutation $\pi$ as in Section~\ref{sec:necklaces}, the original maps apply to the word $w$ obtained from the one-line notation of $\pi^{-1}$ by replacing entries $1,\dots,s_1$ with $a_1$, entries $s_1+1,\dots,s_2$ with $a_2$, and so on. 
Once the set $S$ is fixed, these two descriptions are equivalent.
For the map $\Phi$, one recovers $\pi^{-1}$ by taking the standard permutation of $w$, obtained by numbering the occurrences of $a_1$ in $w$ from left to right, then the occurrences of $a_2$ from left to right, and so on. For the map $\Xi$, one recovers $\pi^{-1}$ by taking the costandard permutation of $w$, obtained similarly by numbering the occurrences of each letter from right to left instead. 

In Section~\ref{sec:GF}, we interpreted multisets of necklaces in $\M_n$ as words in $\W_n$, by identifying each primitive necklace with its Lyndon word, and viewing the multiset as a Lyndon factorization. This identification allows us to interpret the original bijection of Gessel and Reutenauer \cite[Lem.~3.4]{GR} as a weight-preserving bijection on words, $\Phi_\W:\W_n\to\W_n$, which takes the cycle structure of the standard permutation of $w$ to the partition given by the lengths of the Lyndon factors of its image $\hat{w}$. 
In the special case of words of weight $x_1x_2\dots x_n$, which can be viewed as permutations in $\S_n$ since each letter appears once, the restricted bijection $\S_n\to\S_n$ is equivalent to Foata's first fundamental transformation \cite[Sec.~10.2]{Lothaire}, up to trivial symmetries. More specifically, the image of $\pi\in\S_n$ is obtained by writing $\pi$ in cycle form, requiring now that each cycle starts with its smallest element and cycles are ordered by decreasing first element, and then removing the parentheses to get the one-line notation of~$\hat\pi$. In the other direction, the left-to-right minima of $\hat\pi$ determine the initial elements of the cycles of~$\pi$. In the general bijection $\Phi_\W:\W_n\to\W_n$, the role of the left-to-right minima of $\pi\in\S_n$ is played by the left-to-right minima of the sequence of suffixes of $w\in\W_n$, which, as described in Lemma~\ref{lem:LRmin}, determine its Lyndon factorization.

In the special case $S=[n-1]$, the map $f_S$ from Theorem~\ref{thm:main} is a bijection between $\So_n$ and $\Se_n$. 
This bijection is the restriction of $\Psi:\Wo_n\to\We_n$ to the case of words of weight $x_1x_2\dots x_n$, viewed as permutations, conjugated by the above transformation $\pi\mapsto\hat\pi$.
We remark that this map $f_{[n-1]}:\So_n\to\Se_n$ is different from B\'ona's bijection described in Section~\ref{sec:intro} (even after simple variations such as writing each cycle starting with its smallest element), and it does not seem to behave well with respect to ascents and descents in general.
For example, to compute the image of $\pi=(6)(1,7,3,8,4,2,5)\in\So_{8}$, we consider the corresponding word $w=6|1738425\in\Wo_8$, and apply $\Psi$ as shown in the table below.
Then, the word $\Psi(w)=46|38|1725\in\We_8$ corresponds to the permutation $f_{[7]}(\pi)=(4,6)(3,8)(1,7,2,5)\in\Se_8$.
$$\begin{array}{lcr}
O && E\\ \hline
6|17384\dashline25 && - \\
6|17\dashline384&^\cS& 25 \\
6|38\dashline4&^\cP& 1725 \\
6|4&^\cP& 381725 \\
- &^\cF& 46381725
\end{array}$$

\subsection*{Acknowledgments} 
The author thanks Ron Adin and Yuval Roichman for sharing the problem that led to this paper, and Stoyan Dimitrov for providing useful comments and relevant references.

\end{document}